\documentclass{article}

\usepackage{amsmath}
\usepackage{amsthm}

\usepackage[pdftex]{hyperref}
\usepackage{alltt}

\newtheorem{theorem}{Theorem}[section]

\numberwithin{equation}{section}

\title{Some Generalized Harmonic Number Identities}
\author{M.J. Kronenburg}

\begin{document}

\maketitle

\begin{abstract}
Summation by parts is used to find the sum of a finite series of generalized harmonic numbers
involving a specific polynomial or rational function.
The Euler-Maclaurin formula for sums of powers is used to find
the sums of some finite series of generalized harmonic numbers involving nonnegative integer powers,
which can be used to evaluate the sums of the finite series of generalized harmonic numbers involving polynomials.
Many examples and a computer program are provided.
\end{abstract}


\section{Harmonic Numbers and Sums of Powers}

Let the generalized harmonic number for nonnegative $n$, complex order $m$
and complex offset $c$ be defined as:
\begin{equation}\label{genharmonicdef}
 H_{c,n}^{(m)} = \sum_{k=1}^n \frac{1}{(c+k)^m}
\end{equation}
For ease of notation the following convention is used:
\begin{equation}\label{harmonicdef}
 H_n^{(m)} = H_{0,n}^{(m)}
\end{equation}
The traditional harmonic numbers are a special case:
\begin{equation}\label{harmonictrad}
 H_n = H_n^{(1)}
\end{equation}
The following sum of powers can be expressed as a generalized harmonic number
for nonnegative integer $p$:
\begin{equation}\label{powsumdef}
 \sum_{k=1}^n k^p = H_{n}^{(-p)}
\end{equation} 
Let the $B_n^+$ be the Bernoulli numbers with $B_1^+=1/2$ instead of $B_1=-1/2$:
\begin{equation}\label{berndef}
 B_n^+ = (-1)^n B_n = B_n + \delta_{n,1}
\end{equation}
Then application of the Euler-Maclaurin summation formula \cite{V06}
to the sum (\ref{powsumdef}) yields for nonnegative integer $p$:
\begin{equation}\label{faulhaber}
 H_{n}^{(-p)} = \frac{1}{p+1}\sum_{k=1}^{p+1}\binom{p+1}{k}B_{p-k+1}^+ n^k
\end{equation}
This formula expresses any power sum (\ref{powsumdef}) with nonnegative integer $p$
in a polynomial in $n$ of degree $p+1$,
and is used below to obtain a similar reduction for some finite series
of generalized harmonic numbers.
For a list of the resulting expressions of formula (\ref{faulhaber}) up to $p=5$,
see section \ref{examples} below.
In all following derivations it is assumed that $0^0=1$ \cite{GKP94,K97}.

\section{Summation by Parts}

For application in section \ref{sect} below summation by parts is defined and proved.

\begin{theorem}
Let $a\leq b$ be two integers and let $\{x_k\}$ and $\{y_k\}$ be sequences of complex numbers.
Let $\{s_k\}$ be the sequence of complex numbers defined by:
\begin{equation}\label{sdef}
 s_k = \sum_{i=a}^k x_i
\end{equation}
Then there is the following summation by parts formula:
\begin{equation}\label{sumparts}
 \sum_{k=a}^{b-1} x_k y_k = s_{b-1} y_b - \sum_{k=a}^{b-1} s_k ( y_{k+1} - y_k )
\end{equation}
\end{theorem}
\begin{proof}
From (\ref{sdef}) follows:
\begin{equation}
 x_k = s_k - s_{k-1}
\end{equation}
Then (\ref{sumparts}) becomes:
\begin{equation}
\begin{split}
 s_{b-1} y_b & = \sum_{k=a}^{b-1} [ ( s_k - s_{k-1} ) y_k + s_k ( y_{k+1} - y_k ) ] \\
  & = \sum_{k=a}^{b-1} ( s_k y_{k+1} - s_{k-1} y_k ) \\
  & = s_{b-1} y_b - s_{a-1} y_a
\end{split}
\end{equation}
Since $s_{a-1}=0$ by definition (\ref{sdef}), the theorem is proved.
\end{proof}
\vspace{5mm}

\section{Generalized Harmonic Number Identities}\label{sect}

The following identity is a finite series of generalized harmonic numbers involving a certain
polynomial or rational function.
\begin{theorem}
For nonnegative integer $n$ and complex $m$, $w$:
\begin{equation}
 \sum_{k=0}^n [ (k+1)^w - k^w ] H_k^{(m)} = (n+1)^w H_n^{(m)} - H_n^{(m-w)}
\end{equation}
\end{theorem}
\begin{proof}
The summation by parts theorem (\ref{sumparts}) is applied with \mbox{$a=0$}, \mbox{$b=n+1$},
$x_k=(k+1)^w-k^w$ and $y_k=H_k^{(m)}$,
where the $s_k$ become:
\begin{equation}
 s_k = \sum_{i=0}^k [ (i+1)^w - i^w ] = (k+1)^w
\end{equation}
and (\ref{sumparts}) becomes:
\begin{equation}
\begin{split}
 \sum_{k=0}^n [ (k+1)^w - k^w ] H_k^{(m)} 
    & = (n+1)^w H_{n+1}^{(m)} - \sum_{k=0}^n (k+1)^w (k+1)^{-m} \\
    & = (n+1)^w H_{n+1}^{(m)} - H_{n+1}^{(m-w)} \\
    & = (n+1)^w H_n^{(m)} - H_n^{(m-w)}
\end{split}
\end{equation}
\end{proof}
Let the following sum be given \cite{GKP94,K93,SLL89,S90}:
\begin{equation}
 \sum_{k=1}^n \frac{1}{k} H_k = \frac{1}{2} (H_n^2 + H_n^{(2)})
\end{equation}
Then applying this theorem for $m=1$ and $w=-1$
and $H_n^{(m)}=H_{n+1}^{(m)}-1/(n+1)^m$ yields:
\begin{equation}
 \sum_{k=0}^n \frac{1}{k+1} H_k = \frac{1}{2} (H_{n+1}^2 - H_{n+1}^{(2)})
\end{equation}

The following identities are finite series of generalized harmonic numbers
involving nonnegative integer powers, which can be used to evaluate
finite series of generalized harmonic numbers involving polynomials \cite{SLL89,S90}.

\begin{theorem}
For nonnegative integer $n$, $p$  and complex $m$:
\begin{equation}\label{harmsumf}
\begin{split}
 & \sum_{k=0}^n k^p H_k^{(m)} = F(n,p,m) \\
 & = H_n^{(-p)} H_n^{(m)} + H_n^{(m-p)}
  - \frac{1}{p+1} \sum_{k=1}^{p+1} \binom{p+1}{k} B_{p-k+1}^+ H_n^{(m-k)}
\end{split}
\end{equation}
\end{theorem}
\begin{proof}
The order of summation over $k$ and $i$ is reversed:
\begin{equation}
 \sum_{k=1}^n k^p H_k^{(m)} = \sum_{k=1}^n k^p \sum_{i=1}^k \frac{1}{i^m}
  = \sum_{i=1}^n \frac{1}{i^m} \sum_{k=i}^n k^p
\end{equation}
The sum over $k$ is splitted into three parts:
\begin{equation}
 \sum_{k=i}^n k^p = \sum_{k=1}^n k^p + i^p - \sum_{k=1}^i k^p
\end{equation}
The last sum is expressed by (\ref{faulhaber}), and the order
of summation over $k$ and $i$ is reversed again.
\end{proof}

For a list of the resulting expressions of formula (\ref{harmsumf}) up to $p=5$
and $m=4$ see section \ref{examples} below.

\begin{theorem}
For nonnegative integer $n$, $p$  and complex $m$:
\begin{equation}
\begin{split}\label{harmsumg}
 & \sum_{k=0}^n k^p H_{n-k}^{(m)} = G(n,p,m) = H_n^{(-p)} H_n^{(m)} + \delta_{p,0} H_n^{(m)}\\
 & + \frac{1}{p+1} \sum_{k=1}^{p+1} (-1)^k \binom{p+1}{k} 
       [ B_{p-k+1}^+ + (p-k+1)H_n^{(k-p)} ] H_n^{(m-k)}
\end{split}
\end{equation}
\end{theorem}
\begin{proof}
The order of summation over $k$ and $i$ is reversed:
\begin{equation}
 \sum_{k=1}^n k^p H_{n-k}^{(m)} = \sum_{k=1}^n k^p \sum_{i=1}^{n-k} \frac{1}{i^m}
  = \sum_{i=1}^n \frac{1}{i^m} \sum_{k=1}^{n-i} k^p
\end{equation}
The sum over $k$ is evaluated with (\ref{faulhaber}):
\begin{equation}
 \sum_{k=1}^{n-i} k^p
  = \frac{1}{p+1} \sum_{k=1}^{p+1} \binom{p+1}{k} B_{p-k+1}^+ (n-i)^k
\end{equation}
On the right side the power is evaluated with the binomial theorem:
\begin{equation}
 (n-i)^k = \sum_{l=0}^k (-1)^l \binom{k}{l} n^{k-l} i^l
\end{equation}
The trinomial revision of binomial coefficients is used:
\begin{equation}
 \binom{p+1}{k}\binom{k}{l} = \binom{p+1}{l}\binom{p-l+1}{k-l}
\end{equation}
The terms with $l=0$ are splitted off, the order of summation over $k$ and $l$
is reversed, and $k$ and $l$ are interchanged:
\begin{equation}
 \sum_{k=1}^{n-i} k^p
 = \frac{1}{p+1} \sum_{k=1}^{p+1} \binom{p+1}{k} ( B_{p-k+1}^+ n^k + 
   (-1)^k i^k \sum_{l=k}^{p+1} \binom{p-k+1}{l-k} B_{p-l+1}^+ n^{l-k} )
\end{equation}
The first term is evaluated with (\ref{faulhaber}),
and for the second term the summation range over $l$ is changed and (\ref{faulhaber}) is used:
\begin{equation}
  \sum_{l=0}^{p-k+1}\binom{p-k+1}{l}B_{p-k-l+1}^+ n^l = B_{p-k+1}^+ + (p-k+1) H_n^{(k-p)}
\end{equation}
The order of summation over $k$ and $i$ is reversed again.
\end{proof}

For a list of the resulting expressions of formula (\ref{harmsumg}) up to $p=5$
and $m=3$ see section \ref{examples} below.

\section{Generalized Harmonic Numbers Identities with Nonzero Offsets}

The formulas above can be used to derive identities for generalized
harmonic numbers (\ref{genharmonicdef}) with nonnegative integer offset $c=s$.
For this the following equation is used with nonnegative integer $s$:
\begin{equation}\label{offset1}
 H_{s,n}^{(m)} = H_{s+n}^{(m)} - H_s^{(m)}
\end{equation}
which leads to:
\begin{equation}\label{offset2}
\begin{split}
 \sum_{k=0}^n k^p H_{s,k}^{(m)} & = \sum_{k=0}^n k^p ( H_{s+k}^{(m)} - H_s^{(m)} ) \\
 & = -H_s^{(m)} ( H_n^{(-p)} + \delta_{p,0} ) + \sum_{k=0}^n k^p H_{s+k}^{(m)}
\end{split}
\end{equation}
The last sum can be evaluated as follows:
\begin{equation}\label{offset3}
\begin{split}
 \sum_{k=0}^n k^p H_{s+k}^{(m)} & = \sum_{k=s}^{n+s} (k-s)^p H_k^{(m)} \\
 & = \sum_{k=0}^{n+s} (k-s)^p H_k^{(m)} - \sum_{k=0}^{s-1} (k-s)^p H_k^{(m)}
\end{split}
\end{equation}
The power can be expanded with the binomial theorem,
and changing the order of summation results with $F(n,p,m)$ from (\ref{harmsumf}) in:
\begin{equation}\label{offset5}
 \sum_{k=0}^n k^p H_{s+k}^{(m)} = \sum_{k=0}^p (-1)^k \binom{p}{k} s^k
  \left[ F(n+s,p-k,m) - F(s-1,p-k,m) \right]
\end{equation}
which can be substituted in (\ref{offset2}).
In the case $s=0$ the only nonzero term in this formula is for $k=0$ with $s^k=0^0=1$,
and because $F(-1,p,m)=0$ the result in this case is $F(n,p,m)$ as expected.\\

For a list of the resulting expressions of formula (\ref{offset5}) with $s=n$ up to $p=5$
and $m=2$ and with $s=2n$ up to $p=5$ and $m=1$ see section \ref{examples} below.

Similar to the derivation above the second type of sums can be evaluated:
\begin{equation}\label{offset6}
\begin{split}
 \sum_{k=0}^n k^p H_{s,n-k}^{(m)} & = \sum_{k=0}^n k^p ( H_{s+n-k}^{(m)} - H_s^{(m)} ) \\
 & = -H_s^{(m)} ( H_n^{(-p)} + \delta_{p,0} ) + \sum_{k=0}^n k^p H_{s+n-k}^{(m)}
\end{split}
\end{equation}
The last sum can be evaluated as follows:
\begin{equation}\label{offset7}
 \sum_{k=0}^n k^p H_{s+n-k}^{(m)} 
  = \sum_{k=0}^{n+s} k^p H_{s+n-k}^{(m)} - \sum_{k=0}^{s-1} (n+1+k)^p H_{s-1-k}^{(m)}
\end{equation}
The power can be expanded with the binomial theorem,
and changing the order of summation results with $G(n,p,m)$ from (\ref{harmsumg}) in:
\begin{equation}\label{offset9}
 \sum_{k=0}^n k^p H_{s+n-k}^{(m)} = G(n+s,p,m)
 - \sum_{k=0}^p \binom{p}{k} (n+1)^{p-k} G(s-1,k,m)
\end{equation}
which can be substituted in (\ref{offset6}).

For a list of the resulting expressions of formula (\ref{offset9}) with $s=n$ up to $p=5$
and $m=1$ see section \ref{examples} below.
\vspace{5mm}

\section{Examples}\label{examples}

Using the computer program of section \ref{program} below,
a list of many resulting expressions of the main formulas is provided.\\
Formula (\ref{faulhaber}):
\begin{equation}
 H_n^{(0)} = n
\end{equation}
\begin{equation}
 H_n^{(-1)} = \frac{1}{2}n(n+1)
\end{equation}
\begin{equation}
 H_n^{(-2)} = \frac{1}{6}n(n+1)(2n+1)
\end{equation}
\begin{equation}
 H_n^{(-3)} = \frac{1}{4}n^2(n+1)^2
\end{equation}
\begin{equation}
 H_n^{(-4)} = \frac{1}{30}n(n+1)(2n+1)(3n^2+3n-1)
\end{equation}
\begin{equation}
 H_n^{(-5)} = \frac{1}{12}n^2(n+1)^2(2n^2+2n-1)
\end{equation}
\vspace*{3mm}
Formula (\ref{harmsumf}):
\begin{equation}
 \sum_{k=0}^n H_k = (n+1)H_{n+1} - (n+1)
\end{equation}
\begin{equation}
 \sum_{k=0}^n k H_k = H_n^{(-1)} H_{n+1} - \frac{1}{4}n(n+1)
\end{equation}
\begin{equation}
 \sum_{k=0}^n k^2 H_k = H_n^{(-2)} H_{n+1} - \frac{1}{36}n(n+1)(4n+5)
\end{equation}
\begin{equation}
 \sum_{k=0}^n k^3 H_k = H_n^{(-3)} H_{n+1} - \frac{1}{48}n(n+1)(n+2)(3n+1)
\end{equation}
\begin{equation}
 \sum_{k=0}^n k^4 H_k = H_n^{(-4)} H_{n+1} - \frac{1}{1800}n(n+1)(72n^3+243n^2+167n-32)
\end{equation}
\begin{equation}
 \sum_{k=0}^n k^5 H_k = H_n^{(-5)} H_{n+1} - \frac{1}{720}n(n+1)(n+2)(2n+1)(10n^2+19n-9)
\end{equation}
\begin{equation}
 \sum_{k=0}^n H_k^{(2)} = (n+1)H_{n+1}^{(2)} - H_{n+1}
\end{equation}
\begin{equation}
 \sum_{k=0}^n k H_k^{(2)} = H_n^{(-1)}H_{n+1}^{(2)} + \frac{1}{2}H_{n+1} - \frac{1}{2}(n+1)
\end{equation}
\begin{equation}
 \sum_{k=0}^n k^2 H_k^{(2)} = H_n^{(-2)}H_{n+1}^{(2)} - \frac{1}{6}H_{n+1} - \frac{1}{6}(n+1)(n-1)
\end{equation}
\begin{equation}
 \sum_{k=0}^n k^3 H_k^{(2)} = H_n^{(-3)}H_{n+1}^{(2)} - \frac{1}{24}n(n+1)(2n+1)
\end{equation}
\begin{equation}
 \sum_{k=0}^n k^4 H_k^{(2)} = H_n^{(-4)}H_{n+1}^{(2)} + \frac{1}{30}H_{n+1} - \frac{1}{60}(n+1)(n+2)(3n^2-n+1)
\end{equation}
\begin{equation}
 \sum_{k=0}^n k^5 H_k^{(2)} = H_n^{(-5)}H_{n+1}^{(2)} - \frac{1}{360}n(n+1)(12n^3+33n^2+7n-7)
\end{equation}
\begin{equation}
 \sum_{k=0}^n H_k^{(3)} = (n+1)H_{n+1}^{(3)} - H_{n+1}^{(2)}
\end{equation}
\begin{equation}
 \sum_{k=0}^n k H_k^{(3)} = H_n^{(-1)}H_{n+1}^{(3)} + \frac{1}{2}H_{n+1}^{(2)} - \frac{1}{2}H_{n+1}
\end{equation}
\begin{equation}
 \sum_{k=0}^n k^2 H_k^{(3)} = H_n^{(-2)}H_{n+1}^{(3)} - \frac{1}{6}H_{n+1}^{(2)} + \frac{1}{2}H_{n+1} - \frac{1}{3}(n+1)
\end{equation}
\begin{equation}
 \sum_{k=0}^n k^3 H_k^{(3)} = H_n^{(-3)}H_{n+1}^{(3)} - \frac{1}{4}H_{n+1} - \frac{1}{8}(n+1)(n-2)
\end{equation}
\begin{equation}
 \sum_{k=0}^n k^4 H_k^{(3)} = H_n^{(-4)}H_{n+1}^{(3)} + \frac{1}{30}H_{n+1}^{(2)} - \frac{1}{60}(n+1)(4n^2-n+2)
\end{equation}
\begin{equation}
 \sum_{k=0}^n k^5 H_k^{(3)} = H_n^{(-5)}H_{n+1}^{(3)} + \frac{1}{12}H_{n+1} - \frac{1}{24}(n+1)(n+2)(n^2-n+1)
\end{equation}
\begin{equation}
 \sum_{k=0}^n H_k^{(4)} = (n+1)H_{n+1}^{(4)} - H_{n+1}^{(3)}
\end{equation}
\begin{equation}
 \sum_{k=0}^n k H_k^{(4)} = H_n^{(-1)}H_{n+1}^{(4)} + \frac{1}{2}H_{n+1}^{(3)} - \frac{1}{2}H_{n+1}^{(2)}
\end{equation}
\begin{equation}
 \sum_{k=0}^n k^2 H_k^{(4)} = H_n^{(-2)}H_{n+1}^{(4)} - \frac{1}{6}H_{n+1}^{(3)} + \frac{1}{2}H_{n+1}^{(2)} - \frac{1}{3}H_{n+1}
\end{equation}
\begin{equation}
 \sum_{k=0}^n k^3 H_k^{(4)} = H_n^{(-3)}H_{n+1}^{(4)} - \frac{1}{4}H_{n+1}^{(2)} + \frac{1}{2}H_{n+1} - \frac{1}{4}(n+1)
\end{equation}
\begin{equation}
 \sum_{k=0}^n k^4 H_k^{(4)} = H_n^{(-4)}H_{n+1}^{(4)} + \frac{1}{30}H_{n+1}^{(3)} - \frac{1}{3}H_{n+1} - \frac{1}{10}(n+1)(n-3)
\end{equation}
\begin{equation}
 \sum_{k=0}^n k^5 H_k^{(4)} = H_n^{(-5)}H_{n+1}^{(4)} + \frac{1}{12}H_{n+1}^{(2)} - \frac{1}{36}(n+1)(2n^2-2n+3)
\end{equation}
Formula (\ref{harmsumg}):
\begin{equation}
 \sum_{k=0}^n H_{n-k} = (n+1)H_{n+1} - (n+1)
\end{equation}
\begin{equation}
 \sum_{k=0}^n k H_{n-k} = H_n^{(-1)} H_{n+1} - \frac{3}{4}n(n+1)
\end{equation}
\begin{equation}
 \sum_{k=0}^n k^2 H_{n-k} = H_n^{(-2)} H_{n+1} - \frac{1}{36}n(n+1)(22n+5)
\end{equation}
\begin{equation}
 \sum_{k=0}^n k^3 H_{n-k} = H_n^{(-3)} H_{n+1} - \frac{1}{48}n(n+1)(25n^2+13n-2)
\end{equation}
\begin{equation}
 \sum_{k=0}^n k^4 H_{n-k} = H_n^{(-4)} H_{n+1} - \frac{1}{1800}n(n+1)(822n^3+693n^2-133n-32)
\end{equation}
\begin{equation}
 \sum_{k=0}^n k^5 H_{n-k} = H_n^{(-5)} H_{n+1} - \frac{1}{240}n(n+1)(98n^4+116n^3-21n^2-19n+6)
\end{equation}
\begin{equation}
 \sum_{k=0}^n H_{n-k}^{(2)} = (n+1)H_{n+1}^{(2)} - H_{n+1}
\end{equation}
\begin{equation}
 \sum_{k=0}^n k H_{n-k}^{(2)} = H_n^{(-1)} H_{n+1}^{(2)} - \frac{1}{2}(2n+1) H_{n+1} + \frac{1}{2}(n+1)
\end{equation}
\begin{equation}
 \sum_{k=0}^n k^2 H_{n-k}^{(2)} = H_n^{(-2)} H_{n+1}^{(2)} - \frac{1}{6}(6n^2+6n+1) H_{n+1} + \frac{1}{6}(n+1)(5n+1)
\end{equation}
\begin{equation}
 \sum_{k=0}^n k^3 H_{n-k}^{(2)} = H_n^{(-3)} H_{n+1}^{(2)} - 3 H_n^{(-2)} H_{n+1} + \frac{13}{24}n(n+1)(2n+1)
\end{equation}
\begin{equation}
\begin{split}
 \sum_{k=0}^n k^4 H_{n-k}^{(2)} & = H_n^{(-4)} H_{n+1}^{(2)} - \frac{1}{30}(30n^4+60n^3+30n^2-1) H_{n+1} \\
  & + \frac{1}{60}(n+1)(77n^3+65n^2+n-2)
\end{split}
\end{equation}
\begin{equation}
\begin{split}
 \sum_{k=0}^n k^5 H_{n-k}^{(2)} & = H_n^{(-5)} H_{n+1}^{(2)} - 5 H_n^{(-4)} H_{n+1} \\
  & + \frac{1}{360}n(n+1)(522n^3+633n^2+37n-67)
\end{split}
\end{equation}
\begin{equation}
 \sum_{k=0}^n H_{n-k}^{(3)} = (n+1)H_{n+1}^{(3)} - H_{n+1}^{(2)}
\end{equation}
\begin{equation}
 \sum_{k=0}^n k H_{n-k}^{(3)} = H_n^{(-1)} H_{n+1}^{(3)} - \frac{1}{2}(2n+1) H_{n+1}^{(2)} + \frac{1}{2}H_{n+1}
\end{equation}
\begin{equation}
 \sum_{k=0}^n k^2 H_{n-k}^{(3)} = H_n^{(-2)} H_{n+1}^{(3)} - \frac{1}{6}(6n^2+6n+1) H_{n+1}^{(2)} + \frac{1}{2}(2n+1)H_{n+1} - \frac{1}{3}(n+1)
\end{equation}
\begin{equation}
\begin{split}
 \sum_{k=0}^n k^3 H_{n-k}^{(3)} & = H_n^{(-3)} H_{n+1}^{(3)} - 3 H_n^{(-2)} H_{n+1}^{(2)} + \frac{1}{4}(6n^2+6n+1)H_{n+1} \\
 & - \frac{1}{8}(n+1)(7n+2)
\end{split}
\end{equation}
\begin{equation}
\begin{split}
 \sum_{k=0}^n k^4 H_{n-k}^{(3)} & = H_n^{(-4)} H_{n+1}^{(3)} - \frac{1}{30}(30n^4+60n^3+30n^2-1) H_{n+1}^{(2)} \\
  & + n(n+1)(2n+1) H_{n+1} - \frac{1}{60}(n+1)(94n^2+59n+2)
\end{split}
\end{equation}
\begin{equation}
\begin{split}
 & \sum_{k=0}^n k^5 H_{n-k}^{(3)} = H_n^{(-5)} H_{n+1}^{(3)} - 5 H_n^{(-4)} H_{n+1}^{(2)} \\
 & + \frac{1}{12}(30n^4+60n^3+30n^2-1) H_{n+1} - \frac{1}{24}(n+1)(57n^3+57n^2+5n-2)
\end{split}
\end{equation}
Formula (\ref{offset5}):
\begin{equation}
 \sum_{k=0}^n H_{n+k} = (2n+1)H_{2n+1} - n H_n - (n+1)
\end{equation}
\begin{equation}
 \sum_{k=0}^n k H_{n+k} = H_n^{(-1)} H_n + \frac{1}{4}n(n+1)
\end{equation}
\begin{equation}
 \sum_{k=0}^n k^2 H_{n+k} = H_n^{(-2)} [2H_{2n+1}-H_n] - \frac{1}{36}n(n+1)(10n+11)
\end{equation}
\begin{equation}
 \sum_{k=0}^n k^3 H_{n+k} = H_n^{(-3)} H_n + \frac{1}{48}n(n+1)(7n^2+7n-2)
\end{equation}
\begin{equation}
 \sum_{k=0}^n k^4 H_{n+k} = H_n^{(-4)} [2H_{2n+1}-H_n] 
   - \frac{1}{1800}n(n+1)(282n^3+603n^2+257n-92)
\end{equation}
\begin{equation}
 \sum_{k=0}^n k^5 H_{n+k} = H_n^{(-5)} H_n + \frac{1}{720}n(n+1)(74n^4+148n^3+7n^2-67n+18)
\end{equation}
\begin{equation}
 \sum_{k=0}^n H_{n+k}^{(2)} = (2n+1)H_{2n+1}^{(2)} - n H_n^{(2)} - H_{2n+1} + H_n
\end{equation}
\begin{equation}
 \sum_{k=0}^n k H_{n+k}^{(2)} = H_n^{(-1)} H_n^{(2)} + \frac{1}{2}(2n+1)[H_{2n+1}-H_n] - \frac{1}{2}(n+1)
\end{equation}
\begin{equation}
\begin{split}
 \sum_{k=0}^n k^2 H_{n+k}^{(2)} & = H_n^{(-2)}[ 2 H_{2n+1}^{(2)} - H_n^{(2)} ] - \frac{1}{6}(6n^2+6n+1)[H_{2n+1}-H_n] \\
 & + \frac{1}{6}(n+1)(3n+1)
\end{split}
\end{equation}
\begin{equation}
 \sum_{k=0}^n k^3 H_{n+k}^{(2)} = H_n^{(-3)} H_n^{(2)} + 3 H_n^{(-2)}[H_{2n+1}-H_n] 
 - \frac{1}{24}n(n+1)(14n+13)
\end{equation}
\begin{equation}
\begin{split}
 \sum_{k=0}^n k^4 H_{n+k}^{(2)} & = H_n^{(-4)}[ 2 H_{2n+1}^{(2)} - H_n^{(2)} ] \\
  & - \frac{1}{30}(30n^4+60n^3+30n^2-1)[H_{2n+1}-H_n] \\
  & + \frac{1}{60}(n+1)(35n^3+47n^2+5n-2)
\end{split}
\end{equation}
\begin{equation}
\begin{split}
 \sum_{k=0}^n k^5 H_{n+k}^{(2)} & = H_n^{(-5)} H_n^{(2)} + 5 H_n^{(-4)}[H_{2n+1}-H_n] \\
  & - \frac{1}{360}n(n+1)(74n+67)(3n^2+3n-1)
\end{split}
\end{equation}
\begin{equation}
 \sum_{k=0}^n H_{2n+k} = (3n+1) H_{3n+1} - 2n H_{2n} - (n+1)
\end{equation}
\begin{equation}
 \sum_{k=0}^n k H_{2n+k} = -\frac{1}{2}n(3n+1) H_{3n+1} + n(2n+1) H_{2n} + \frac{3}{4}n(n+1)
\end{equation}
\begin{equation}
\begin{split}
 \sum_{k=0}^n k^2 H_{2n+k} & = \frac{1}{2}n(2n+1)(3n+1) H_{3n+1} - \frac{1}{3}n(2n+1)(4n+1) H_{2n} \\
 & - \frac{1}{36}n(n+1)(40n+17)
\end{split}
\end{equation}
\begin{equation}
\begin{split}
 \sum_{k=0}^n k^3 H_{2n+k} & = -\frac{1}{4}n^2(3n+1)(5n+3) H_{3n+1} + n^2(2n+1)^2 H_{2n} \\
 & + \frac{1}{48}n(n+1)(77n^2+45n-2)
\end{split}
\end{equation}
\begin{equation}
\begin{split}
 \sum_{k=0}^n k^4 H_{2n+k} & = \frac{1}{10}n(2n+1)(3n+1)(11n^2+5n-1) H_{3n+1} \\
 & - \frac{1}{15} n(2n+1)(4n+1)(12n^2+6n-1) H_{2n} \\
 & - \frac{1}{1800}n(n+1)(4692n^3+4143n^2+467n-152)
\end{split}
\end{equation}
\begin{equation}
\begin{split}
 \sum_{k=0}^n k^5 H_{2n+k} & = -\frac{1}{4}n^2(3n+1)(14n^3+16n^2+3n-1) H_{3n+1} \\
 & + \frac{1}{3} n^2(2n+1)^2(8n^2+4n-1) H_{2n} \\
 & + \frac{1}{240}n(n+1)(1036n^4+1148n^3+177n^2-87n+6)
\end{split}
\end{equation}
Formula (\ref{offset9}):
\begin{equation}
 \sum_{k=0}^n H_{2n-k} = (2n+1)H_{2n+1} - n H_n - (n+1)
\end{equation}
\begin{equation}
 \sum_{k=0}^n k H_{2n-k} = n(2n+1)H_{2n+1} - \frac{1}{2}n(3n+1) H_n - \frac{5}{4}n(n+1)
\end{equation}
\begin{equation}
\begin{split}
 \sum_{k=0}^n k^2 H_{2n-k} & = \frac{1}{3}n(2n+1)(4n+1)H_{2n+1} - \frac{1}{6}n(2n+1)(7n+1) H_n \\
 & - \frac{1}{36}n(n+1)(64n+11)
\end{split}
\end{equation}
\begin{equation}
\begin{split}
 \sum_{k=0}^n k^3 H_{2n-k} & = n^2(2n+1)^2 H_{2n+1} - \frac{1}{4}n^2(3n+1)(5n+3) H_n \\
 & - \frac{1}{48}n(n+1)(131n^2+51n-2)
\end{split}
\end{equation}
\begin{equation}
\begin{split}
 \sum_{k=0}^n k^4 H_{2n-k} & = \frac{1}{15}n(2n+1)(4n+1)(12n^2+6n-1) H_{2n+1} \\
 & - \frac{1}{30}n(2n+1)(93n^3+66n^2+2n-1) H_n \\
 & - \frac{1}{1800}n(n+1)(7932n^3+4953n^2-43n-92)
\end{split}
\end{equation}
\begin{equation}
\begin{split}
 \sum_{k=0}^n k^5 H_{2n-k} & = \frac{1}{3}n^2(2n+1)^2(8n^2+4n-1) H_{2n+1} \\
 & - \frac{1}{4}n^2(3n+1)(14n^3+16n^2+3n-1) H_n \\
 & - \frac{1}{720}n(n+1)(5308n^4+4604n^3+221n^2-251n+18)
\end{split}
\end{equation}

\section{Computer Program}\label{program}

The Mathematica$^{\textregistered}$ \cite{W03} program used to compute the expressions
is given below.
\begin{alltt}
Unprotect[Power];
0^0=1;
Protect[Power];
BernPlus[n_]:=BernoulliB[n]+KroneckerDelta[n,1]
PowSum[p_]:=
 Factor[1/(p+1)Sum[Binomial[p+1,k]BernPlus[p-k+1]n^k,\{k,1,p+1\}]]
HarmFun[p_]:=If[p>0,HarmonicNumber[n+1,p]-1/(n+1)^p,PowSum[-p]]
HarmSumF[p_,m_]:=Simplify[HarmFun[-p]HarmFun[m]+HarmFun[m-p]
 -1/(p+1)Sum[Binomial[p+1,k]BernPlus[p-k+1]HarmFun[m-k],\{k,1,p+1\}]]
HarmSumG[p_,m_]:=Simplify[HarmFun[-p]HarmFun[m]
 +KroneckerDelta[p,0]HarmFun[m]+1/(p+1)Sum[(-1)^k Binomial[p+1,k]
 (BernPlus[p-k+1]+(p-k+1)HarmFun[k-p])HarmFun[m-k],\{k,1,p+1\}]]
HarmSumF[p_,m_,s_]:=Simplify[Sum[(-1)^k Binomial[p,k]s^k
 (ReplaceAll[HarmSumF[p-k,m],n->n+s]
 -ReplaceAll[HarmSumF[p-k,m],n->s-1]),\{k,0,p\}]]
HarmSumG[p_,m_,s_]:=Simplify[ReplaceAll[HarmSumG[p,m],n->n+s]
 -Sum[Binomial[p,k](n+1)^(p-k)ReplaceAll[HarmSumG[k,m],n->s-1],
 \{k,0,p\}]]
HarmTable[m_]:=Table[HarmonicNumber[n+1,i],\{i,m\}]
HarmTable[m_,s_Integer]:=Table[HarmonicNumber[n+s+1,i],\{i,m\}]
HarmTable[m_,s_]:=Table[HarmonicNumber[s+Mod[i+1,2](n+1),
 Quotient[i+1,2]],\{i,2m\}]
HarmonicSumF[p_,m_]:=Module[\{t=HarmTable[m],u\},
 u=Factor[CoefficientArrays[HarmSumF[p,m],t]];
 u[[1]]+Dot[u[[2]],t]]
HarmonicSumG[p_,m_]:=Module[\{t=HarmTable[m],u\},
 u=Factor[CoefficientArrays[HarmSumG[p,m],t]];
 u[[1]]+Dot[u[[2]],t]]
HarmonicSumF[p_,m_,s_]:=Module[\{t=HarmTable[m,s],u\},
 u=Factor[CoefficientArrays[HarmSumF[p,m,s],t]];
 u[[1]]+Dot[u[[2]],t]]
HarmonicSumG[p_,m_,s_]:=Module[\{t=HarmTable[m,s],u\},
 u=Factor[CoefficientArrays[HarmSumG[p,m,s],t]];
 u[[1]]+Dot[u[[2]],t]]

(* Compute some examples *)
HarmonicSumF[4,3]//TraditionalForm
HarmonicSumF[2,2,n]//TraditionalForm
HarmonicSumF[3,1,2n]//TraditionalForm
HarmonicSumG[3,2]//TraditionalForm
HarmonicSumG[3,1,n]//TraditionalForm
HarmonicSumF[2,1,0]//TraditionalForm
HarmonicSumF[2,1,3]//TraditionalForm
HarmonicSumF[2,1,s]//TraditionalForm
\end{alltt}

\pdfbookmark[0]{References}{}

\bibliographystyle{amsplain}

\end{document}